\documentclass[12pt]{amsart}
\usepackage{amsfonts,amsmath,amssymb,amsthm,graphicx,enumerate,color,stmaryrd}
\usepackage[english]{babel}
\usepackage[utf8]{inputenc}
\usepackage[T1]{fontenc}
\usepackage{framed}
\usepackage{tikz}
\usetikzlibrary{arrows}

\usepackage{array}

\usepackage[margin=1.6cm]{geometry}

\newcommand{\ds}{\displaystyle}

\newcommand{\defeq}{\mathrel{\mathop:}=}

\newcommand{\B}{\mathbb{B}}

\newcommand{\E}{\mathbb{E}}

\renewcommand{\L}{\mathbb{L}}
\newcommand{\M}{\mathbb{M}}
\newcommand{\N}{\mathbb{N}}
\renewcommand{\P}{\mathbb{P}}

\newcommand{\R}{\mathbb{R}}

\newcommand{\Z}{\mathbb{Z}}
\newcommand{\indic}{{\bf 1}}
\newcommand{\ind}[1]{{\indic_{\{#1\}}}}

\newcommand{\Ut}{{\widetilde U}}
\newcommand{\Vt}{{\widetilde V}}

\newcommand{\ut}{{\widetilde u}}

\newcommand{\Ar}{\mathcal A}

\newcommand{\dm}{\begin{pmatrix}} 
\newcommand{\fm}{\end{pmatrix}}
\newcommand{\ddm}{\begin{vmatrix}} 
\newcommand{\fdm}{\end{vmatrix}}

\usepackage{textcomp,eurosym}

\newcommand{\s}{\,|\,} 

\newcommand{\st}{\,:\,} 

\usepackage{mathtools}
\DeclarePairedDelimiter\abs{\lvert}{\rvert} 
\DeclarePairedDelimiter\nor{\|}{\|} 
\DeclarePairedDelimiter\pa{(}{)} 
\DeclarePairedDelimiter\cro{[}{]} 

\newtheorem{theorem}{Theorem}
\newtheorem{lemma}{Lemma}
\newtheorem{proposition}{Proposition}
\newtheorem{corollary}{Corollary}
\newtheorem*{remark*}{Remark}

\newtheorem{definition}{Definition}

\author[A. Collevecchio]{Andrea Collevecchio}
\address{School of Mathematical Sciences, Monash  University, Melbourne} \email{Andrea.Collevecchio@monash.edu}

\author[K. Hamza]{Kais Hamza}
\address{School of Mathematical Sciences, Monash  University, Melbourne} \email{Kais.Hamza@monash.edu}

\author[L. Tournier]{Laurent Tournier}
\address{LAGA, Universit\'e Paris 13, Sorbonne Paris Cit\'e, CNRS, UMR 7539, 93430 Villetaneuse, France.
 }
\email{tournier@math.univ-paris13.fr}

\title[]{A deterministic walk on the randomly oriented\\ Manhattan lattice}

\DeclareMathOperator*{\thepath}{Path}
\DeclareMathOperator*{\thediam}{diam}

\begin{document}

\maketitle

{\footnotesize \noindent{\slshape{\bfseries Abstract.} Consider a randomly-oriented two dimensional Manhattan lattice where each horizontal line and each vertical line is assigned, once and for all, a random direction by flipping independent and identically distributed coins. A deterministic walk is then started at the origin and at each step moves diagonally to the nearest vertex in the direction of the horizontal and vertical lines of the present location.  This definition can be generalized, in a natural way, to larger dimensions, but  we  mainly focus on the two dimensional case. In this context the process localizes on two vertices at all large times, almost surely.  We also provide estimates for the tail of the length of paths, when the walk is defined on the two dimensional lattice. In particular, the probability of the path to be larger than $n$ decays sub-exponentially in $n$. It is easy to show that higher dimensional paths may not localize on two vertices but will still eventually become periodic, and are therefore bounded.}
}
\bigskip

\section{Introduction}
   \label{s:intro}

Random walks receive significant attention in the probability literature. The source of randomness can be associated to the walker, the environment in which she evolves or both.
When the environment is random, most of the effort has concentrated on the independent and identically distributed (i.i.d.) case. In this paper, we study a deterministic walk in a random environment with infinite length correlation. The setting is that of the randomly oriented Manhattan lattice in $\mathbb{Z}^d$, in which lines that are parallel to the axes are uniformly and independently oriented. Ledger et al.~\cite{Led} study the simple random walk on this oriented lattice, i.e.\ a walker that chooses her next step uniformly at random amongst the $d$ coordinate directions and with the orientation prescribed by the environment. They prove bounds that point to the super-diffusivity of this process in dimensions 2 and 3 but the question of transience remains an open problem. See \cite{Cam1,Nad1,Nad2} for more precise results on a partially oriented lattice, and \cite{Cam1, Ler1} for the role and importance of such models in quantum statistical physics.

Our goal is to study a deterministic process that moves diagonally on the randomly oriented Manhattan lattice by following the local drift of the previous walk, i.e.\ in a direction that combines all $d$ orientations at the present vertex: From a vertex $x \in \Z^d$, the walk travels to the vertex $y$ where, for $i=1,\ldots,d$,  $y_i-x_i=\pm1$ depending on the orientation of the line parallel to the $i$th axis and going through $x$.

In our analysis we focus mostly on the two dimensional case, which is formally similar to the so-called corner percolation of Pete \cite{pete2008}. In the latter the walk alternates between horizontal and vertical steps, following the random orientation of the corresponding line. It is shown (\cite{pete2008}) that the path forms a finite cycle, almost surely, and that the distribution of its diameter has a polynomial tail, with an explicit exponent. Let us remark that such a path can also be viewed as a beam of light among mirrors placed on vertices of $\Z^2$ and rotated according to the random line orientations; such random mirror models (see e.g.\ \cite{Kozma} for the usual i.i.d.\ setting), as discrete analogs of Lorentz gas, are another important instance of deterministic motion in random environment.

In contrast to corner percolation, we prove that in dimension 2 the above defined process eventually localizes on exactly two vertices. We also study the tail of the length of the path starting from the origin and show that it has as a stretched exponential decay. While the fact that paths are eventually periodic is rather natural, and expected for wide classes of deterministic walks, the property that the only cycles are trivial ones (i.e.\ of length 2) seems to single this model out. Our approach to prove this property is combinatorial in nature, and does not rely on the i.i.d.\ structure of the orientation of lines.

Let us finally note that deterministic walks in random environment can also be viewed, by ``integrating the environment out'', i.e.\ under the annealed measure, as random walks with a long memory, with the further property that they are automatically self avoiding (before possibly cycling). A remarkable example in that respect is the exploration path of the frontier of a percolation cluster, subject to some boundary conditions (see for instance~\cite{CamiaNewman} for a celebrated result of convergence to an ${\rm SLE}_6$ process).


\subsection{Model and main results}\label{mod}
Let $U=(U_y)_{y\in\Z}$ and $V=(V_x)_{x\in\Z}$ be sequences of random variables in $\{-1,1\}$ defined on a probability space $(\Omega,\Ar,\P)$. $U$ and $V$ can be thought of as orientations of respectively horizontal and vertical lines in $\Z^2$.

We assume all variables $U_x$, $V_y$, $x,y\in\Z^2$, to be independent and have symmetric distribution.

Given $(U,V)$, we are interested in the discrete path $Z=(Z_n)_{n\ge0}=((X_n,Y_n))_{n\ge0}$ in $\Z^2$ that starts at $Z_0=(0,0)$ and follows diagonals given by the orientations $(U,V)$:
\[\text{for all $n\ge0$,}\qquad Z_{n+1}=Z_n+(U_{Y_n},V_{X_n}).\]
Note that all the randomness of $Z$ is contained in the data of $(U,V)$.

\begin{figure}[h!]
	\includegraphics[width=7cm]{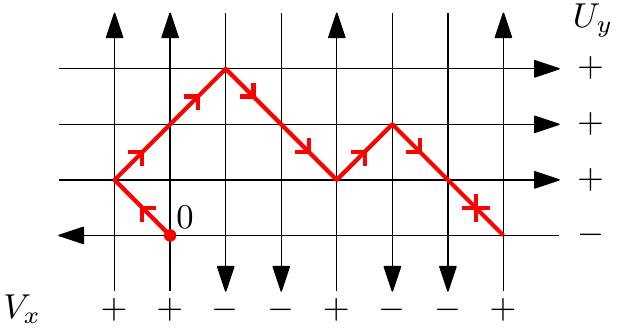}
\caption{Example of orientation of lines and the associated path $Z$, eventually trapped}
\end{figure}

Our main result is the following.

\begin{theorem}\label{thm:Z}
	Almost surely, the path $Z$ gets eventually stuck on one edge, i.e.\ it is eventually 2-periodic.
\end{theorem}

Since $Z$ is a deterministic function of $(U,V)$, it is actually useful to see this result as a more general property of the graph $\omega$ made of the union of such paths started at all initial locations in $\Z^2$.

Define $\L=(\Z^2,\E)$ as the directed graph whose vertex set is $\Z^2$ and whose edge set $\E$ consists of all diagonal edges across faces, i.e.\ $((x,y),(x',y'))$ with $\abs{x-x'}=\abs{y-y'}=1$.
Given $(U,V)$, define $\omega$ to be the following subset of $\E$:
\[\omega=\Big\{\big((x,y),(x+U_y,y+V_x)\big)\st (x,y)\in\Z^2\Big\}\subset\E.\]
Notice that $\omega$ contains exactly one (directed) outgoing edge at each vertex, and that $Z$ is the only path in $\omega$ starting at $(0,0)$.

The above theorem expresses that almost surely, every path in $\omega$ ends with a \emph{trap}, i.e.\ a pair of opposite edges ($U_y+U_{y+V_x}=0$ and $V_x+V_{x+U_y}=0$). The content is twofold: paths in $\omega$ don't leave to infinity, and the only cycles in $\omega$ are made of two edges.
We can be a little more precise about $\omega$ by also excluding paths coming from infinity:

\begin{theorem}\label{thm:omega}
Almost surely, every connected component of $\omega$ is bounded and contains one  trap.
\end{theorem}

In other words, the connected components of $\omega$ are finite trees rooted at a trap. Note that we will commonly view subsets of $\E$ as graphs, meaning implicitly that their vertex set is the set of endpoints of the edges they contain, which will usually be the whole set $\Z^2$.


\subsection{Extension of the model to $\Z^3$}
The model admits a natural extension to $\Z^3$, where every east-west, north-south and up-down line is randomly oriented. Thus one needs random variables $(U_{(y,z)})_{(y,z)\in\Z^2}$, $(V_{(x,z)})_{(x,z)\in\Z^2}$ and $(W_{(x,y)})_{(x,y)\in\Z^2}$ to define $Z=(Z_n)_{n\ge0}$ by $Z_0=(0,0,0)$ and, for $n\ge0$,
\[Z_{n+1}=Z_n+(U_{(Y_n,Z_n)},V_{(Z_n,X_n)},W_{(X_n,Y_n)}).\]
We still assume all orientations to be independent and symmetric. The graph $\omega$ is defined similarly as before.

\begin{theorem}
Almost surely, the path $Z$ gets eventually periodic.

Almost surely, every component of $\omega$ is bounded.
\end{theorem}

However, in contrast to the two-dimensional case, nontrivial cycles exist in higher dimension, as illustrated on Figure~\ref{fig:3D}.

\begin{figure}[h!]
	\includegraphics[width=7cm]{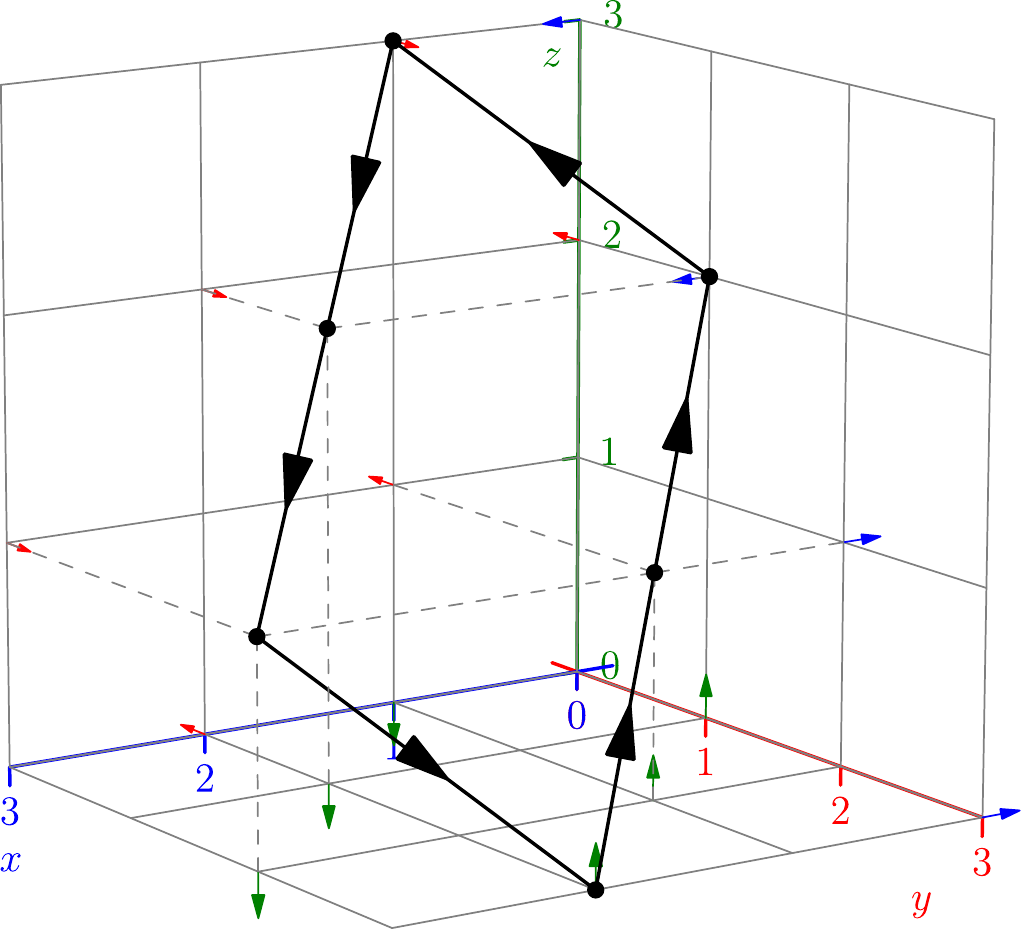}
\caption{Example of a nontrivial cycle in $\Z^3$ (in black). The colored arrows indicate the orientations of the lines relevant to the cycle.\label{fig:3D}}
\end{figure}

\subsection{Definitions: Traps, sources and crossings}

Let us come back to $\Z^2$. The proof relies on geometrical properties of the graph $\omega$ and it will be convenient to introduce a set of definitions regarding this graph, both to write the formal proof and to first describe its principle.

We already noticed that $\omega$ may contain traps, which are pairs of opposite edges that cross diagonally a face of the square lattice $\Z^2$. Let us classify more generally the different types of faces.

\begin{definition}
A \emph{cell} is a face of the square lattice $\Z^2$; it can be identified with its center $(x+\frac12,y+\frac12)$ for some $(x,y)\in\Z^2$.
We say that an edge in $\E$ \emph{crosses} a given cell if both its endpoints are corners of that cell.

A cell $(x+\frac12,y+\frac12)$ is a \emph{source} (resp.\ a \emph{crossing}, resp.\ a \emph{trap}) in $\omega$ if it is crossed by no edge (resp.\ by one edge, resp.\ by two edges) in $\omega$.
\end{definition}

This definition is illustrated on Fig.~\ref{fig:cells}. One should in particular notice that, from the definition of $\omega$, the two edges crossing a trap cell are always opposite edges, hence indeed corresponding to what we earlier called a trap, and that no more than two edges may cross a cell, hence each cell is either a source, a crossing or a trap.

\begin{figure}[h!]
	\includegraphics[width=7cm,page=4]{all_paths.pdf}
\hspace{1cm}
	\includegraphics[width=7.8cm]{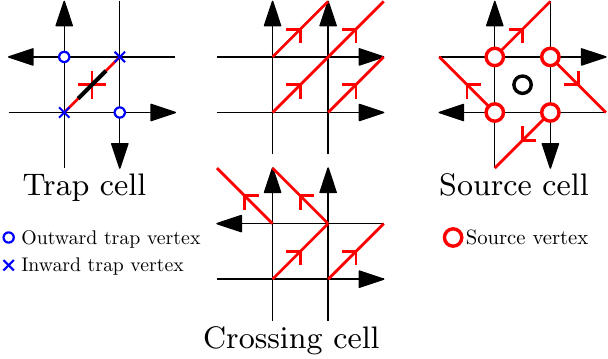}
\caption{Representation (left) of all paths (i.e., of $\omega$) in a box. As reminded on the right, black circles indicate source cells, black diagonal strokes indicate trap cells; Red circles indicate source vertices, blue circles indicate inward trap vertices and blue crosses indicate outward trap vertices. These three types of vertices can be cumulated. \label{fig:cells}}
\end{figure}

\begin{definition}
A vertex $v\in\Z^2$ is a \emph{source vertex} if it is a corner of a source cell. If $v$ is a corner of a trap cell, it is an \emph{inward trap vertex} if it is one of two endpoints of the crossing edges, and an \emph{outward trap vertex} otherwise.
\end{definition}

Note that not every vertex has one among these qualifiers, and that they are not mutually exclusive. Indeed, source cells and trap cells may lie next to each other. However, an important tool in our analysis is a \emph{reduced model} that is precisely made so as to avoid this situation (see Section~\ref{sec:reduction}).

\newcommand{\IV}{\mathfrak{trap}}

We will always view the edges in $\omega$ as embedded into $\R^2$ in the following way: every edge is represented by a straight line between its endpoints except for traps, which are embedded as two disjoint arcs in the same cell. Although $\E$ is not planar, using this embedding we have the following deterministic statement, which simply follows by inspection of each type of cell:

\begin{lemma}
$\omega$ is a planar graph.
\end{lemma}

Let us also remind the useful fact that $\E$ has two connected components, given by odd and even vertices in $\Z^2$ respectively, each of which is planar and isomorphic to the square lattice. 
\subsection{Organization of the paper}
The next section introduces a certain transformation of the graph that plays a technical role by preventing trap and source cells from being neighbours. This is instrumental in Section~\ref{sec:trivial_cycles} to prove the combinatorial property that the only cycles in $\omega$ are the traps. Section~\ref{sec:bounded} finally proves that the components of $\omega$ are bounded almost surely and obtains estimates on the distribution of their size.


\section{Reduced model}\label{sec:reduction}

In this section we argue that, in proving Theorem~\ref{thm:Z}, we may replace the i.i.d.\ distributions of $U$ and $V$ by a distribution that is still shift invariant and symmetric, but also such that the alternating patterns $(+1,-1,+1)$ and $(-1,+1,-1)$ don't appear anymore. The model where $U$ and $V$ follow this new distribution will be called the \emph{reduced model}, in the sense that it will intuitively be obtained by removing all such patterns from the initial $U$ and $V$ wherever they occur.

\subsection{Local reduction}

Let us first justify that if, in the sequence of orientations of horizontal lines, we replace one occurrence of $(-1,+1,-1)$ by $(-1)$ (i.e.\ we remove two lines and shift the next ones), then connectivities are preserved, hence in particular cycles and infinite paths remain (shorter) cycles and infinite paths. See Figure~\ref{fig:reduction} for an illustration.

\begin{figure}[h]
	\includegraphics[width=8cm]{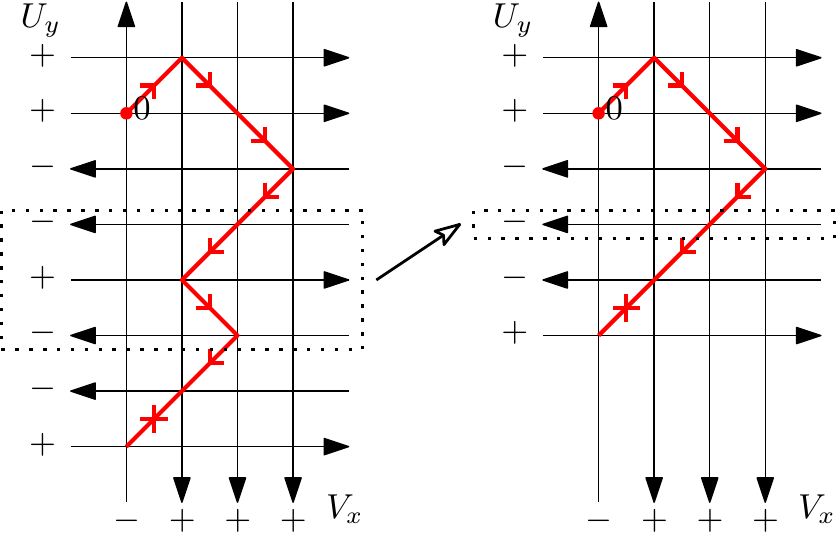}
	\caption{Reduction of a configuration: the occurrence of the alternating pattern $(-1,+1,-1)$ on horizontal lines is replaced by a single $-1$, which has a the only effect on paths of ``removing a zigzag pattern''. } \label{fig:reduction}
\end{figure}

Let us first consider occurrence of patterns $-1,+1,-1$ or $+1,-1,+1$ at zero: define
\[C_0=\{(u_k)_{k\in\Z}\st (u_{-1},u_0,u_1)\in\{(-1,+1,-1),(+1,-1,+1)\}\}.\]


Define the bijection $\sigma_0:\Z\to\Z\setminus\{0,1\}$ by $\sigma_0(k)=k$ if $k<0$ and $\sigma_0(k)=k+2$ if $k\ge0$, and its inverse $\pi_0:\Z\setminus\{0,1\}\to\Z$. Then we let, for $u\in\{-1,+1\}^\Z$,  $\ut=(u_{\sigma_0(n)})_{n\in\Z}$, which corresponds to deleting the values $u_0$ and $u_1$ and shifting the next ones.


By extension, if $a=(x,y)\in\Z^2$ and $x\notin\{0,1\}$, we define $\pi_0(a)=(\pi_0(x),y)$.

Finally, if $V\in C_0$ (where $V=(V_x)_{x\in\Z}$ gives the orientations of vertical lines), we define $\pi_0(\omega)$ in the same way as $\omega$, using $\Vt=(V_{\sigma_0(x)})_{x\in\Z}$ instead of $V$.

For $a,b\in\Z^2$, we shall write ``$a\to b$ in $\omega$'' if there is a path $a=a_0,a_1,\ldots,a_n=b$ in $\omega$ where $a_1,\ldots,a_{n-1}$ are not inward trap vertices, i.e.\ the path doesn't cross a trap.

\begin{lemma}
If $V\in C_0$, and $a,b\in\Z^2\setminus(\{0,1\}\times\Z)$ are such that $a\to b$ in $\omega$, then $\pi_0(a)\to\pi_0(b)$ in $\pi_0(\omega)$.
\end{lemma}

\begin{proof}
Assume for instance $(V_{-1},V_0,V_1)=(-1,+1,-1)$.
We work by induction on the length of the path from $a$ to $b$ in $\omega$. If $a$ and $b$ are at distance 1, then clearly $\pi_0(a)\to\pi_0(b)$ since the edge $(\pi_0(a),\pi_0(b))$ is in $\pi_0(\omega)$. Assume there is a path $a=(x_0,y_0),a_1=(x_1,y_1),\ldots,(x_n,y_n)=b$ in $\omega$. If $x_{n-1}\notin\{0,1\}$, then by induction $\pi_0(a)\to\pi_0(a_{n-1})$ in $\pi_0(\omega)$ hence a path to $\pi_0(b)$ using the final edge $(\pi_0(a_{n-1}),\pi_0(b))$. Assume $x_{n-1}=1$ (the other case is symmetric). Thus, $x_n=2$ since $x_n=x_{n-1}\pm 1$ and $x_n\notin\{0,1\}$ by assumption; this implies that $U_y=+1$ where $y=y_{n-1}$. If $x_{n-2}=2$, then it means $U_{y+1}=-1$; we can apply the induction to have $\pi_0(a)\to\pi_0(a_{n-2})=(0,y+1)$, and it remains to notice that we still have $(0,y+1)\to (-1,y)\to (0,y-1)=\pi_0(b)$ because $\widetilde{V}(-1)=V(-1)=-1=V(1)$. Assume finally that $x_{n-2}=0$. Since $V(0)=+1$, we must have $y_{n-2}=y-1$ and $U_{y-1}=+1$. This forces $a_{n-3}=(-1,y)$. By induction, $\pi_0(a)\to\pi_0(a_{n-3})$ in $\pi_0(\omega)$. We also have $\pi_0(a_{n-3})=(-1,y)\to(0,y-1)=\pi_0(b)$ hence $\pi_0(a)\to\pi_0(b)$. This concludes the proof.
\end{proof}


More generally, for any $x\in\Z$, we can consider the occurrence of alternating patterns at $x$ by defining $C_x=\theta_x C_0$, and $\pi_x:\Z\setminus\{x,x+1\}\to\Z$ is given by $\pi_x=\theta_x\pi_0\theta_{-x}$.


\subsection{Global reduction}

Next, we define a global reduction by applying the local reduction everywhere possible on $\Z$, so as to obtain sequences $\Ut$ and $\Vt$ that do not contain any alternating pattern.

Note first that the local reduction can be applied repeatedly to an alternating block of even size ``$(-)+-+-+-(+)$'' (where parentheses are used to distinguish certain signs), so as to obtain ``$(-)(+)$'', and to an alternating block of odd size ``$(-)+-+(-)$'' so as to obtain ``$(-)$''; in the latter case one has to choose which of the bounding signs is taken out, although the resulting sequence is the same either way. Let us thus define the set of alternance sites as follows: first, let
\[A=\{x\in\Z\st U_{x-1}\ne U_x\ne U_{x+1}\},\]
and then due to the previous remark define, for $x\in A$,
\[\ell_x=\inf\{y>x\st y\notin A\}-\sup\{y<x\st y\notin A\}-1,\]
the length of the block around $x$,
and the set of sites to be removed as
\[B=A\cup\{x+1\st x\in A,\ x+1\notin A,\ \ell_x\text{ is odd}\}.\]
Note that removing sites in $B$ produces a sequence without alternating pattern. Indeed, first, the blocks in $A$ are separated by definition by blocks of at least two consecutive identical signs; and removing the extra signs in $B\setminus A$ does not create blocks of length one since these odd-length blocks separate identical signs on their left and right.

Let us introduce the bijection $\sigma_B:\Z\to\Z\setminus B$ denoting by $\sigma_B(n)$ the $(n+1)$-th element of $\{0,1,2,\ldots\}$ not in $B$, when $n\ge0$, and the $(-n)$-th element of $\{-1,-2,\ldots\}$ not in $B$ when $n\le-1$, so that $\Ut=(U_{\sigma_B(x)})_{x\in\Z}$ is obtained from $U$ by removing indices in $B$ and shifting the others toward zero in order to ``fill the gaps''. Note that almost surely $\sigma_B(n)\to\pm\infty$ as $n\to\pm\infty$. Denote by $\pi_B:\Z\setminus B\to\Z$ its inverse function.

After removal, i.e.\ in the sequence $\Ut$, the lengths of blocks of consecutive identical signs are i.i.d.\ and integrable (they have a geometric distribution). In particular, up to a random shift this sequence can be made stationary.


%
%

Define as before $\pi_B(a)=(\pi_B(x),y)$ for $a=(x,y)\in\Z^2$ with $x\notin B$, and $\pi_B(\omega)$ by reindexing all edges whose ends have a first component not in $B$.

Recall the notation $a\to b$ in $\omega$ to denote the existence of a path in $\omega$ that connects $b$ to $a$ (i.e. $b$ is accessible from $a$).
By applying the local reduction several times, we immediately obtain:

\begin{lemma}
If $a,b\in\Z^2\setminus(\{0,1\}\times B)$ are such that $a\to b$ in $\omega$, then $\pi_B(a)\to\pi_B(b)$ in $\pi_B(\omega)$.
\end{lemma}

\begin{corollary}
For any nontrivial cycle $\sigma$ in $\omega$, $\pi_B(\sigma)$ is a nontrivial cycle in $\pi_B(\omega)$.

Almost surely, for any infinite path $\gamma$ in $\omega$ starting from $a\notin B$, $\pi_B(\gamma)$ is an infinite path starting from $\pi_B(a)$.

Almost surely, for any infinite path $\gamma$ in $\omega$ ending at $a\notin B$, $\pi_B(\gamma)$ is an infinite path ending at $\pi_B(a)$.

Almost surely, for any bi-infinite path $\gamma$ in $\omega$, $\pi_B(\gamma)$ is a bi-infinite path.
\end{corollary}

\begin{proof}
Due to the previous lemma (local reduction), we only need to check that no degeneracy could happen: that a nontrivial cycle $\sigma$ can be neither completely removed nor reduced to a trivial cycle; and that infinite paths keep being infinite.

For $\pi_B(\sigma)$ to be empty, the cycle $\sigma$ would have to lie entirely on $B\times\{0,1\}$, hence by connectivity to lie inside one block of alternating $+$ and $-$. However a similar analysis as in the previous lemma shows that the only paths in such blocks are trivial cycles or paths that cross in a zig-zag shape and in particular exit the block. 

Furthermore, $\pi_B(\sigma)$ cannot be trivial since at least four edges of $\sigma$ connect a point outside $B$ to another point (possibly in $B$); consider indeed exit points out of a block of $B$. These edges are all kept in $\pi_B(\omega)$ hence $\pi_B(\sigma)$ must have at least four edges.

Finally, if $\gamma$ is an infinite path in $\omega$, then it must stay only finite intervals of time within each block of $B$. Indeed, the only paths in such blocks are trivial traps or zig-zags crossing the blocks and thus exiting at each side. Therefore $\pi_B(\omega)$ is still an infinite path.
\end{proof}


By symmetry, the same procedure can be applied to the second coordinate to define a sequence $\Vt$ out of $V$ by removing alternating patterns. What we shall call the \emph{reduced model} refers to the sequences $\Ut$, $\Vt$. Due to the previous corollary, we conclude that proving Theorem~\ref{thm:Z} for these sequences will imply the result for $U,V$.


\section{Nonexistence of nontrivial cycles}\label{sec:trivial_cycles}

The aim of this section is to prove the following deterministic statement, where we recall that a \emph{cycle} in $\omega$ is a sequence $a_0,\ldots,a_n$ in $\Z^2$ such that $a_n=a_0$ and $(a_k,a_{k+1})\in\omega$ for $k=0,\ldots,n-1$, and a \emph{trap} is a cycle in $\omega$ of length $n=2$.

\begin{proposition}\label{pro:trivial_cycles}
The only cycles in $\omega$ are traps.
\end{proposition}

Let us first give a sketch of the argument. The heuristic idea is that a cycle should contain as many source and trap cells due to a conservation argument: each path inside the cycle has a beginning and an end (or merges into the cycle); and on the other hand, each source cell is the beginning of 4 paths, while each trap cell sees the end of 2 paths and the merging of 4 paths into 2, thus contributing a net loss of 4 paths. However, a parity constraint due to the alternance of up/down and left/right slopes in the path forces these numbers to differ by 1 unit, leading to a contradiction.

The conservation argument is made somewhat less transparent in situations with neighbouring source and trap cells, therefore for this part we will actually work with the reduced model from the previous section.


We start with the parity constraint, which will decompose into two steps. We first consider cycles on a new graph $\M$, for which we prove that any cycle must circumscribe one more source cell than trap cells. We further argue than any nontrivial (non trap) cycle in $\omega$ maps into a cycle of the previous type, in a way that preserves the number of circumscribed source and trap cells.

\subsection{Parity constraint}

Let us first notice that the set of trap and source cells has a simple structure: they are located at the lattice of intersections of the sign change lines (definition follows), and traps and sources alternate in a checkered pattern (cf.\ Figure~\ref{fig:blocks}).

\begin{definition}
The \emph{sign change lines} are the vertical lines $\{x+\frac12\}\times\R$ for which $x$ is such that $V_x\ne V_{x+1}$ and the horizontal lines $\R\times\{y+\frac12\}$ for which $y$ is such that $U_y\ne U_{y+1}$.

The connected components of the complement of this set of lines are called \emph{sign blocks}. Thus, $(U_y,V_x)$ is constant for $(x,y)$ inside a given sign block. 

The \emph{block lattice} $\B$ is the lattice induced by the intersections of the sign change lines, i.e.\ whose vertices are the traps and sources, and whose edges connect nearest neighbours on the sign change lines.
\end{definition}

\begin{figure}[h]
	\includegraphics[width=11cm]{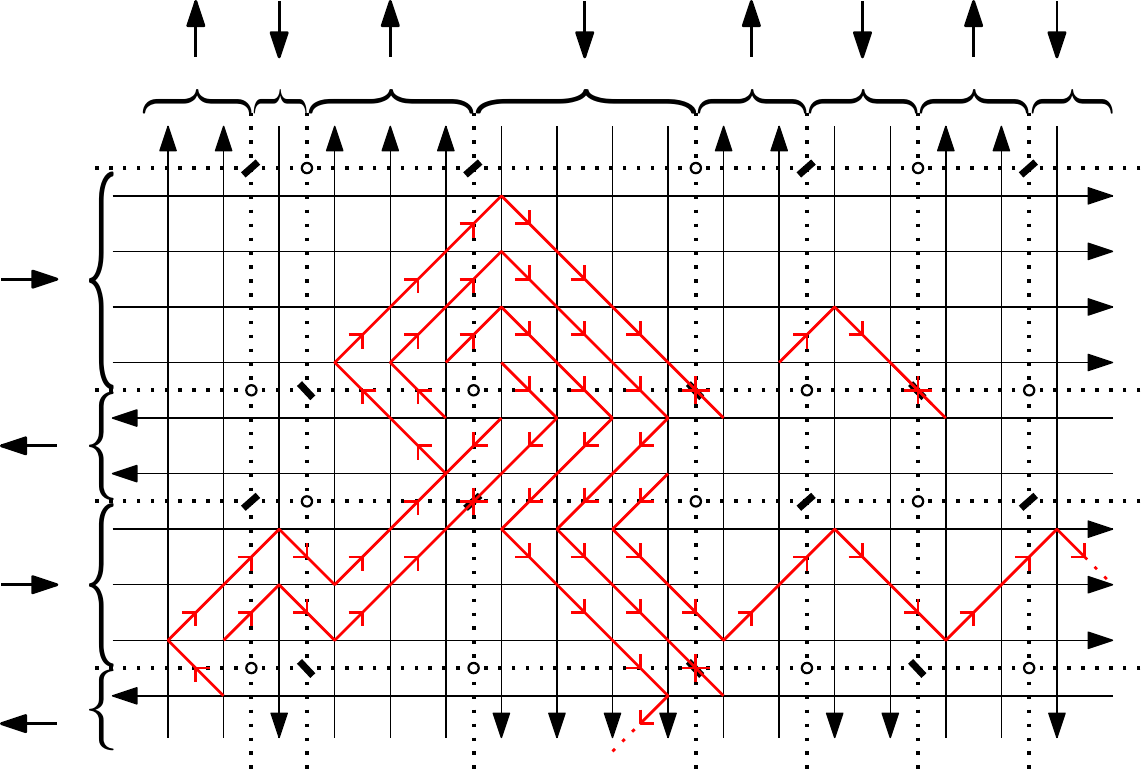}
	\caption{Dotted lines are the sign change lines. Their intersections define the block lattice $\B$, whose vertices are either sources (circles) or traps (diagonal strokes). A few examples of paths are represented in red. } \label{fig:blocks}
\end{figure}

\begin{remark*}
 Although the graph structure of $\B$ is that of $\Z^2$, its embedding in $\R^2$ depends on $\omega$. This is precisely the feature that distinguishes $\B$ from $\Z^2$.
\end{remark*}
\begin{definition}
The \emph{mid-edge graph} $\M$ (see Figure~\ref{fig:network}, left) is a directed (non planar) graph whose vertex set consists of the middle points of the edges of $\B$, and whose edges connect middle points across a face of $\B$ (a ``block'') according to its orientation in the wider sense, i.e.\ according to the following rule (up to symmetries i.e. rotations and reflection):
\begin{center}
	\includegraphics[width=3cm]{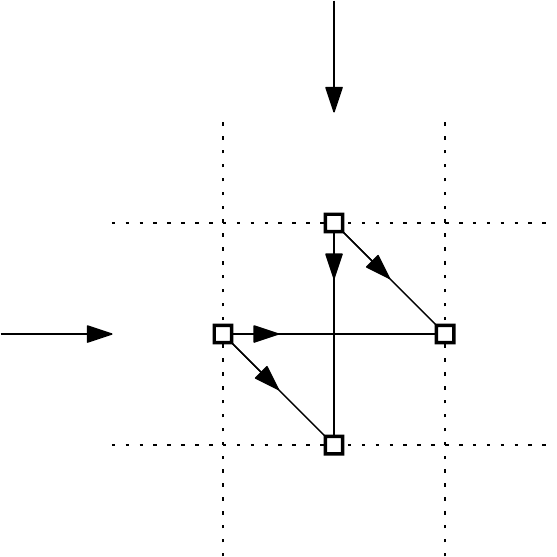}
\end{center}
Thus each vertex of $\M$ has two neighbours, and if a path in $\omega$ crosses two edges of $\B$ consecutively then their middle points are connected in $\M$, in the order of their crossing. See Figure \ref{fig:network}.



To avoid confusions with paths, i.e.\ paths of $\omega$, we call paths in $\M$ \emph{mid-edge paths}. 
And a \emph{mid-edge cycle} is a cycle in $\M$.
\end{definition}

\begin{figure}[h]
	\includegraphics[width=11cm]{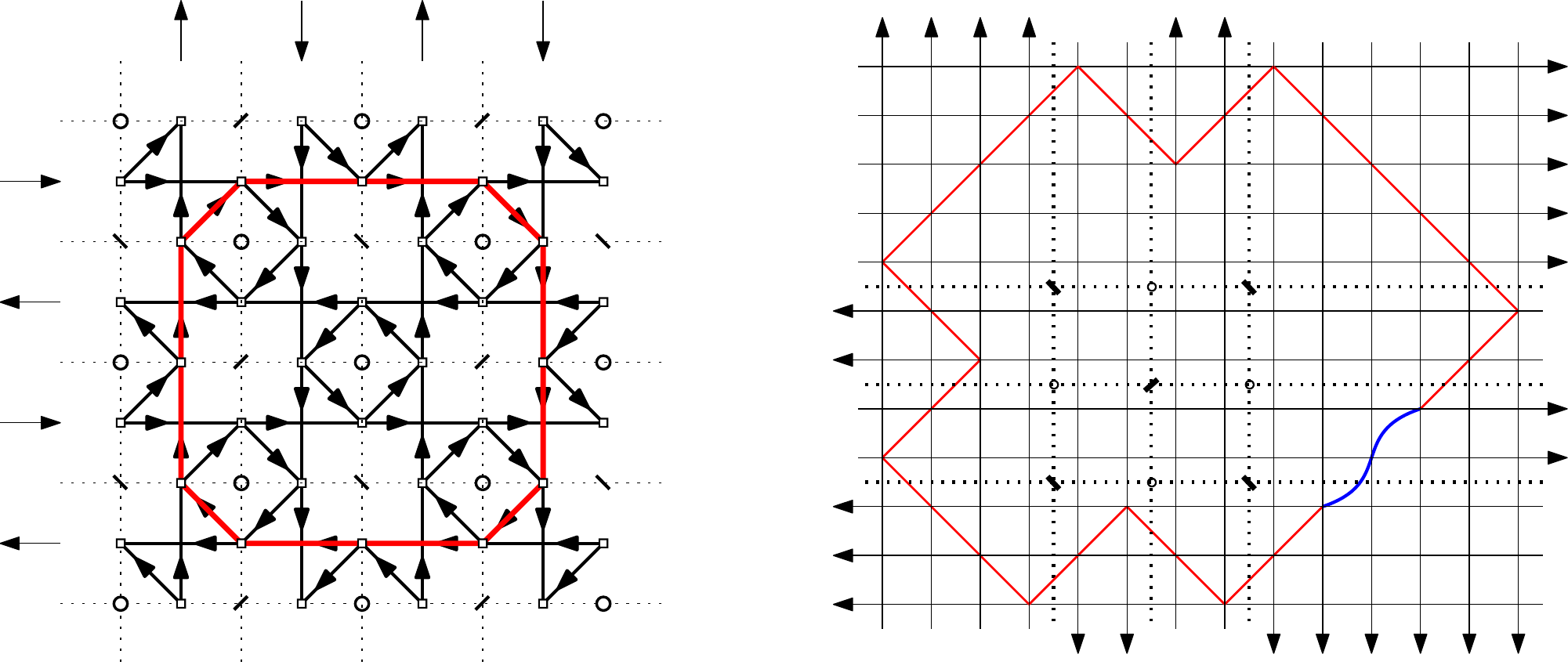}
	\caption{(Left) Mid-edge graph $\M$, and an example of mid-edge cycle, enclosing 5 sources and 4 traps. The dotted lines correspond to the sign change lines, where the signs are suggested by the arrows to the left and top sides of the graph. The square marks are the mid-edge vertices of the block lattice $\B$, i.e.\ the vertices of $\M$. The circles are the sources and the diagonal dashes are the traps. (Right) The corresponding hypothetical cycle in $\omega$ with an ``illegal'' edge (in blue), added for the sake of illustration.}\label{fig:network}
\end{figure}

To a nontrivial cycle in $\omega$ (if there would exist any), we associate a mid-edge path given by the sequence of edges of $\B$ that it crosses: the definition of $\M$ ensures that these middles of edges form a path in $\M$. Note that the original cycle doesn't visit traps or sources hence this is well defined; and that the resulting path is obviously again a cycle.


A cycle cannot cross the same edge of $\B$ more than once. Let us suppose that such a cycle exists and let us assume w.l.o.g. that it crosses a vertical edge of $\B$ from left to right in two locations. Then by planarity, and because it is a simple path, there must be a crossing in the opposite direction (right to left) in-between the left-to-right crossings. This leads to a contradiction and proves that the previous procedure does not introduce multiple edges. The resulting cycle is thus a simple cycle, which can be homeomorphically deduced from the initial cycle. This also ensures that it is planar, i.e.\ without crossings in its simple embedding. By Jordan curve theorem, we may then consider the interior of this cycle.

\begin{lemma}\label{lem:numbers}
Consider a simple planar mid-edge cycle $\sigma$. Denoting by $s$ and $t$ the numbers of sources and traps in the interior of $\sigma$, we have
\[s=t+1.\]
\end{lemma}

\begin{proof}
The proof procedes by induction on the number of vertices of $\B$ (both sources and traps) in the interior of~$\sigma$.

The lemma holds in the case when $\sigma$ surrounds only one vertex of $\B$. Given the graph $\M$, this one vertex indeed has to be a source.

Let $n\in\N$, $n\ge1$, and assume that the lemma holds for simple planar mid-edge cycles containing at most $n$ vertices of $\B$ in their interior. Let $\sigma$ be a simple planar mid-edge cycle containing $n+1$ vertices of $\B$ in its interior.

Up to symmetrizing the argument, we may assume that $\sigma$ turns counterclockwise.

Consider, among the leftmost vertices of $\sigma$, the one that is highest, and denote it by $x$. Due to the location of $x$ and the orientation of $\sigma$, the last step before $x$ was at the north-east of $x$. Regarding the next step after $x$, two cases may occur: either $\sigma$ goes in the south-east or the south direction.

In the case when $\sigma$ goes in the south-east direction after $x$, Figure~\ref{fig:local_modif_1} indicates how to modify locally the path around $x$ to construct a simple planar mid-edge cycle $\sigma'$ that now contains one source and one trap fewer than $\sigma$. By the assumption, $\sigma$ therefore fulfills the lemma. Note that the case of a cycle of length~4 is excluded by the fact that $n+1\ge2$.

\begin{figure}[h]
	\includegraphics[width=11cm,page=2]{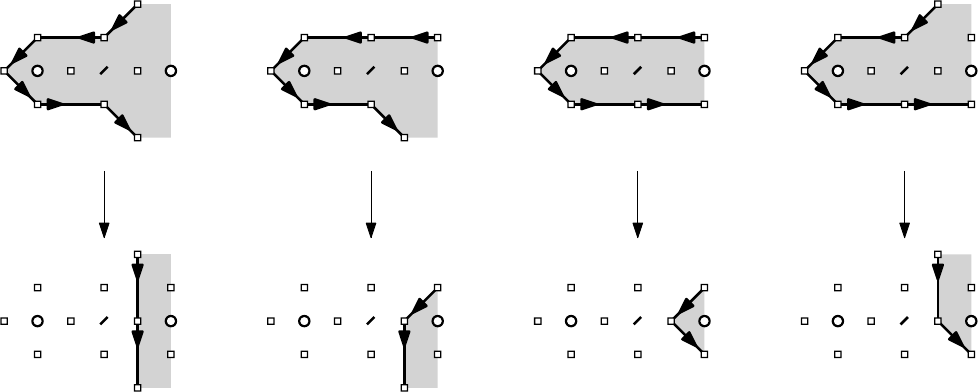}
	\caption{Reduction of the size of the cycle in the proof of Lemma~\ref{lem:numbers}. The top sketches represent the possible shapes of $\sigma$ around the leftmost vertex $x$, while the line below shows the local transformation into a smaller mid-edge cycle $\sigma'$. The grey area represents a part that is necessarily in the interior of $\sigma$, which justifies the possibility of defining $\sigma'$ as shown. }\label{fig:local_modif_1}
\end{figure}

Assume we are in the other case, i.e.\ that $\sigma$ goes south after $x$. If, before $x$, $\sigma$ comes from the north-east and then from the south-east, then we can use again the simplification from Figure~\ref{fig:local_modif_1}, up to a quarter-turn rotation. We may therefore assume that the steps before $x$ were from the north-east and then from the east, whence the different situations are displayed in Figure~\ref{fig:local_modif_2}. Let us further distinguish whether the source that lies south-east from the source next to $x$ is inside or outside of $\sigma$. If it is outside, then the figure shows how to introduce two smaller planar simple mid-edge cycles $\sigma_1$ and $\sigma_2$ whose numbers of sources and traps satisfy by induction $s_1=t_1+1$ and $s_2=t_2+1$, but are also given by $s=s_1+s_2-1$ and $t=t_1+t_2$, which implies the lemma for $\sigma$. If it is inside, then $\sigma$ can be reduced to a simple planar mid-edge cycle $\sigma'$ in the way shown by Figure~\ref{fig:local_modif_2} and since $\sigma'$ contains 2 fewer sources and traps than $\sigma$, the lemma follows as well by induction.

\begin{figure}[h]
	\includegraphics[width=11cm,page=3]{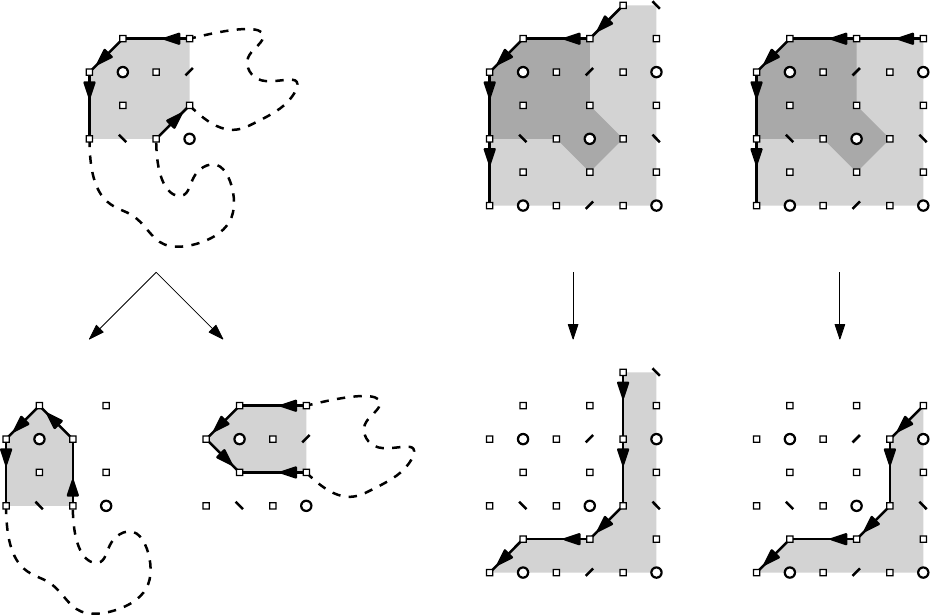}
	\caption{Reduction of the size of the cycle in the proof of Lemma~\ref{lem:numbers}, second case. Here, on the first line the dark grey region is \emph{assumed} to be in the interior of $\sigma$, which forces the light grey region to be as well.}\label{fig:local_modif_2}
\end{figure}
\end{proof}

\subsection{Conservation argument}

Notice that the number of source cells in the interior of a cycle is the same as the number of sources in the interior of the corresponding mid-edge cycle, and similarly for traps. The conclusion of the proof then follows by contradiction by comparing the previous lemma with this one:

\begin{lemma}\label{lem:conservation_reduced}
Any nontrivial cycle in the reduced model must enclose an equal number of source and trap cells.
\end{lemma}

Note that no path crosses a source cell and that the only path that crosses a trap cell is a trivial cycle. Thus, a nontrivial cycle does not cross diagonally any source or trap cells and the number of source and trap cells it encloses is therefore a well-defined integer.

\begin{proof}
Recall that, in the reduced model, each block of consecutive lines sharing the same orientation has size at least 2, so that trap cells and source cells are never contiguous.
Therefore, it always holds true that source vertices are the starts of paths, and that among the four corners of any trap, two of them (the \textit{outward trap vertices}) are points where two paths merge into one, and at the other two (the \textit{inward trap vertices}) there is one path entering the trap and thus ending there. 

As is readily checked on each connected component individually (the components are either binary trees or cycle-rooted binary trees), the number of starts of paths inside the cycle is equal to the total number of ends and of merges of paths inside the cycle plus the number of merges of paths with the cycle from its interior. Starts correspond to source vertices, ends to inward trap vertices, and merges to outward trap vertices. Notice also that, by planarity constraints, when a path merges with the cycle from inside, then the trap cell of this outward trap vertex must lie \textit{inside} the cycle. We conclude from this that the number of source vertices inside the cycle is equal to the number of trap vertices inside the cycle whose cell lies inside the cycle. This proves the lemma.
\end{proof}

Although the previous lemma suffices to prove the main theorem, we sketch a proof of the following slightly stronger result about the original (non-reduced) model. This indeed shows how the initial intuition could be made formal, although the introduction of the reduced model leads to a more transparent proof.

\begin{lemma}\label{lem:conservation}
Any nontrivial cycle in the initial model must enclose an equal number of sources and traps.
\end{lemma}

\begin{proof}
In order to adapt the proof of the previous lemma to the general case, we shall first introduce a decoration of the graph that keeps planarity and removes particular cases. At each source, we add four vertices marking starts of paths, each connected to one of the four corners, cf.~Figure~\ref{fig:modif}; and at each trap, we replace the two diagonal paths by two noncrossing paths ending in two new vertices called sinks.

\begin{figure}[h]
	\includegraphics[width=6cm]{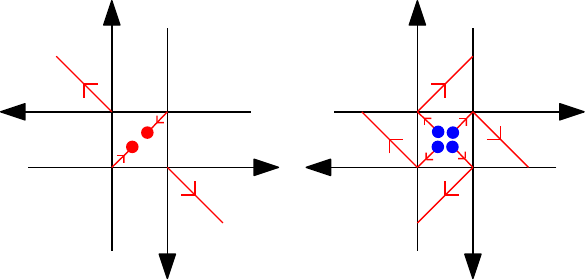}
	\caption{Local modification of paths, in the proof of Lemma~\ref{lem:conservation}: case of a trap (left) and of a source (right)}\label{fig:modif}
\end{figure}

With these modifications, each source always ``creates 4 paths'', in that it contains 4 starts of paths. And each trap always terminates 2 paths at the sinks and produces 2 merges at the other corners, as can be checked in each case on Figure~\ref{fig:cases}. Other merges may also happen at the other two corners of a trap, but only when they are also a merging corner of another trap. Note at last that merges only happen at corners of traps (which may at the same time be corners of sources). Hence, we can apply a counting argument similar to one used in Lemma~\ref{lem:conservation_reduced}. 


\begin{figure}[h]
	\includegraphics[width=11cm,page=1]{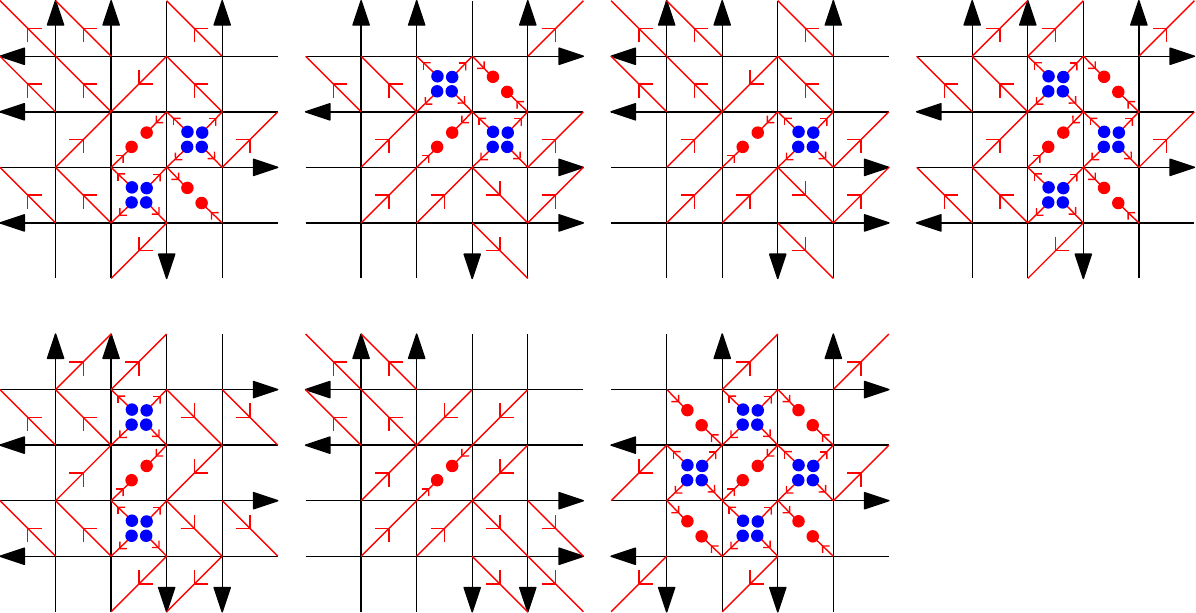}\\
	\caption{List of all configurations of lines surrounding a trap, up to symmetries. Note that in each case, of the four corners around a trap, two lead into the trap and the other two are merging points; the first two corners may also be merging points (in the 2nd, 4th and 7th cases), but only when next to another trap. }\label{fig:cases}
\end{figure}

\end{proof}

\section{Finiteness of paths}\label{sec:bounded}

The following proposition is based on a simple renewal argument. The proof builds on the next three lemmas.

\begin{proposition}\label{prop:bounded}
In any dimension, almost surely, the path starting at vertex $o$ is bounded.
\end{proposition}

Remember $X_0=Y_0=0$. Let $T=\inf\{n\geq0:X_n\not\in[\min_{k\leq n-1}X_k,\max_{k\leq n-1}X_k]\mbox{ and }U_{Y_n}=-1\}$ be the first time the path meets a new vertical line and a horizontal line that points to the left (with the understanding that $\min_{k\leq -1}X_k=+\infty$ and $\max_{k\leq -1}X_k=-\infty$).

\begin{lemma}
On $\{T<+\infty\}$, $X_T\geq0$, and on $\{1\leq T<+\infty\}$, $X_{T-1}=\max_{k\leq T-1}X_k$ and $X_T=X_{T-1}+1$.
Furthermore, $T=\inf\{n\geq0:X_{n+1}=X_n-1\}$.
\end{lemma}
\begin{proof}
Suppose $T<+\infty$ and $X_T<0$. Then $U_{Y_0}=U_0=+1$. Indeed, if $U_{Y_0}=-1$, then $T=0$ and $X_T=0$, contradicting the assumption $X_T<0$. It follows that $X_1=1$. Again, if $U_{Y_1}=-1$, then $T=1$ ($X_1\not\in[\min_{k\leq0}X_k,\max_{k\leq0}X_k]=\{0\}$) and $X_T=1$, again contradicting the assumption $X_T<0$. It follows that $U_{Y_1}=+1$ and $X_2=2$. This argument can be repeated to show that $\forall n$, $U_{Y_n}=+1$ and $X_n=n$; that is $\tau=+\infty$.
\end{proof}

\begin{lemma}
$\mathbb{P}(V_{X_T}=+1,T<+\infty)=1/2$ and more generally, for any $A\in\mathcal{F}_T$, $\mathbb{P}(V_{X_T}=+1,A\cap\{T<+\infty\})=1/2\mathbb{P}(A\cap\{T<+\infty\})$; that is
$$\mathbb{P}(V_{X_T}=+1|\mathcal{F}_T\cap\{T<+\infty\})=
\mathbb{P}(V_{X_T}=-1|\mathcal{F}_T\cap\{T<+\infty\})=1/2.$$
In particular $V_{X_T}1_{T<+\infty}$ is independent of $\mathcal{F}_T\cap\{T<+\infty\}$.
\end{lemma}
\begin{proof}
The proof follows from the following observations.

First, $Z_n$ is $\sigma((U_y)_{|y|\leq n-1},(V_x)_{|x|\leq n-1})$-measurable. It follows that $U_{Y_n}$ is $\sigma((U_y)_{|y|\leq n},(V_x)_{|x|\leq n-1})$-measurable, and $V_{X_n}$ is $\sigma((U_y)_{|y|\leq n-1},(V_x)_{|x|\leq n})$-measurable, and that  $(U_{Y_n},V_{X_n})$ is $\sigma((U_y)_{|y|\leq n},(V_x)_{|x|\leq n})$-measurable.

Then, for $x\geq0$, $\{X_T=x,T<+\infty\}\in\sigma((U_y)_{y\in\mathbb{Z}},(V_a)_{a<x})$ and is independent of $V_x$.
\end{proof}

For any $v\in\mathbb{Z}^2$, we write $v\in\IV$ when $v$ is an inward trap vertex.

\begin{lemma}\label{lem:trap}
$\mathbb{P}(Z_T\in\IV|\mathcal{F}_T\cap\{T<+\infty\})=1/2$.
\end{lemma}
\begin{proof}
This is an immediate application of the above lemma and the fact that $Z_T$ is a trap if and only if either $Y_{T-1}=Y_T+1$ and $V_{X_T}=+1$, or $Y_{T-1}=Y_T-1$ and $V_{X_T}=-1$.
\end{proof}

\subsubsection*{Proof of Proposition \ref{prop:bounded}}

We write the proof for $d=2$ and then explain the (simple) adaptation to higher dimensions.

Fix the horizontal environment such that at least one arrow points to the left and one arrow points to the right (i.e. not all arrows point in the same direction).

Let $T_n$ be the consecutive times the path meets a new vertical line (to the right of $o$) and a horizontal line that points to the left. Set $T_n=+\infty$ if fewer than $n$ such occurrences exist, and $T_0=0$ ($T_1=T$).

It follows from Lemma \ref{lem:trap} that,
$$\P(T_{n+1}=+\infty|{\mathcal{F}}_{T_n},T_n<+\infty) \geq \P(Z_{T_n}\in\IV|{\mathcal{F}}_{T_n},T_n<+\infty) = 1/2.$$
It follows that $\P(T_{n+1}=+\infty,T_n<+\infty)\geq (1/2)\P(T_n<+\infty)$, that $\P(T_{n+1}<+\infty)\leq (1/2)\P(T_n<+\infty)$ and that
$$\P(T_{n+1}<+\infty)\leq (1/2)^n\P(T_1<+\infty).$$
By virtue of the Borel-Cantelli lemma, we conclude that $\P(T_n<+\infty\mbox{ i.o.})=0$ and that $\P(T_n=+\infty\mbox{ i.o.})=1$ (recall that the sequence $T_n$ is increasing).

Now, on the event $\{T_{n-1}<+\infty,T_n=+\infty,X_{T_{n-1}}=x\}$, either $\sup_kX_k\leq x$ and the path is bounded on the right, or $\exists k$ such that $X_k=x+1$. In this case, $U_{x+1}=+1$ (otherwise $T_n<+\infty$) and $X_{k+1}=x+2$. In turn, this implies that $U_{x+2}=+1$ and so on. We end up with a setting where all horizontal lines met after $T_{n-1}$ are pointing to the right, and $Z_{k+\ell}=(X_k+\ell,Y_k+W_\ell)$, where $W_\ell=V_{x+1}+\ldots+V_{x+\ell}$ is a simple symmetric random walk. This is not possible as sooner or later the path must meet a horizontal line that points to the left, proving that such a $k$ cannot exist and that the path is indeed bounded.

The principle of this proof adapts in any dimension: for the same reason, in any direction where the path is unbounded, there must be infinitely many hyperplanes that are entered for the first time at a place where the orthogonal direction is pointing backward, and each time, with probability $1/2^{d-1}$, the other orientations (which are newly discovered) push the path to its previous location, thereby creating a trap. Due to the independence, we again conclude that almost surely this happens eventually.
\hfill$\Box$\bigskip

The non-existence of paths coming from infinity will now follow from a simple application of the mass transportation principle (cf.~\cite{LyonsPeres}) that we first remind:

\begin{lemma}[Mass transport principle]
Let $f:(u,v,\omega)\in\Z^2\times\Z^2\times\Omega\mapsto f(u,v,\omega)\in\R_+$ be a measurable function. We shall henceforth view $\omega\mapsto f(u,v,\omega)$ as random variables and drop $\omega$ from the notation. Assume that $\E[f(u,v)]=\E[f(u+w,v+w)]$ for all $u,v,w\in\Z^2$. 
Then for any vertex $o\in\Z^2$,
\[\E\bigg[\sum_{u\in\Z^2} f(o,u)\bigg]=\E\bigg[\sum_{u\in\Z^2} f(u,o)\bigg].\]
\end{lemma}

\begin{proof}
The assumption yields
\[\sum_{u\in\Z^2}\E[f(o,u)]=\sum_{u\in\Z^2}\E[f(2o-u,o)],\]
and the second sum is equal, up to reindexing to $\sum_{u\in\Z^2} \E[f(u,o)]$.
\end{proof}

Recall, for $a,b\in\Z^2$, we write $a\to b$ in $\omega$ if there is a path $a=a_0,a_1,\ldots,a_n=b$ in $\omega$ and $a_1,\ldots,a_{n-1}\notin\IV$, i.e.\ the path doesn't cross a trap (recall that $\IV$ is the set of vertices of trivial traps).

For every $o\in\Z^2$, if $o\to\infty$, let us denote $\tau(o)=\infty$ and otherwise denote by $\tau(o)$ the only $v\in\IV$ such that $o\to v$, i.e.\ $\tau(o)$ is the trap where the path starting at $o$ ends.

Let us also denote by $C(o)$ the connected component of $o$ in $\omega$:
\[C(o)=\{u\in\Z^2\st u\to o\text{ or }o\to u\}.\]
Note that, from the definition of $\to$, the set $C(o)$ contains at most one vertex in $\IV$.

\begin{proposition}\label{prop:from_infinity}
Almost surely, $\omega$ contains no semi-infinite path ending at a trap vertex. Furthermore,
\[\E\cro{\abs{C(o)}\s o\in\IV}<\infty.\]
\end{proposition}

\begin{proof}
Let $n\in\N$. Let us apply the mass transportation principle to
\[f(u,v)=\ind{v = \tau(u)\ne\infty},\]
which sends unit mass, from every vertex leading to a trap, toward this trap. This gives
\[\P(o\not\to\infty)=\E[\ind{o\in\IV}\#\{u\st u\to o\}]=\E[\ind{o\in\IV}\abs{C(o)}],\]
In particular, we deduce that $\abs{C(o)}<\infty$ a.s.\ if $o\in\IV$, which shows that almost surely no infinite path leads to a trap vertex.
\end{proof}

Now Theorem~\ref{thm:omega} is a consequence of the previous two propositions: almost surely the path starting from $o$ ends in a trap, and almost surely the component of a trap is of integrable size, hence almost surely the connected component of $o$ is finite.

\section{Integrability of components}


We are interested in an estimation of the size of the component of the origin. To that aim, we shall derive a quantitative version of Proposition~\ref{prop:bounded}, which will show that the tail of the distribution of the size of $C(o)$ (either diameter or number of vertices) decays sub-exponentially, as a stretched exponential. As a significant corollary, this will entail integrability of the size of $C(o)$.

Note indeed that, although we proved (cf.\ Proposition~\ref{prop:from_infinity}) $\E[\abs{C(o)}\s o\in\IV]<\infty$, this doesn't imply $\E[\abs{C(o)}]<\infty$ due to a size bias caused by the conditioning. More precisely, we actually need a second moment: we have
\[\E[\abs{C(o)}]=\E[\abs{C(o)}^2,\,o\in\IV],\]
as an instance of a general size biasing principle:

\begin{lemma}\label{lem:size_bias}
For any measurable function $\Phi:\Omega\to\R_+$ such that almost surely $\Phi\circ\tau_u=\Phi$ for all $u\in C(o)$, where $\tau_u:\Omega\to\Omega$ denotes the translation by $u\in\Z^2$,
\[\E[\Phi]=\E[\Phi\cdot\abs{C(o)},\,o\in\IV].\]
\end{lemma}

\begin{proof}
As an application of the mass transport principle to the function $f:(u,v)\mapsto \Phi\circ\tau_u\cdot\ind{u\to v,\,v\in\IV}$,
one gets
\[\sum_u \E[\Phi\circ\tau_u\ind{u\to o,\,o\in\IV}] = \sum_u \E[\Phi\ind{o\to u,\,u\in\IV}].\]
However, due to the assumption on $\Phi$, the left hand side rewrites as
\[\sum_u \E[\Phi\ind{u\to o,\,o\in\IV}]=\E\cro[\bigg]{\Phi\sum_u\ind{u\to o}\ind{o\in\IV}}=\E[\Phi\cdot\abs{C(o)}\ind{o\in\IV}],\]
and it follows from Theorem~\ref{thm:Z} that a.s.\ there is a (clearly unique) $u\in\IV$ such that $o\to u$ hence the right hand side rewrites as
\begin{align*}
\E\cro[\bigg]{\Phi\sum_u\ind{o\to u, u\in\IV}}=\E[\Phi].
\end{align*}
The lemma follows.
\end{proof}

In addition to considering $C(o)$, we are also interested in the trajectory
\[\thepath(o)\defeq \{u\st o\to u\}=\{Z_n\st n\in\N\},\]
and more precisely to the maximum distance reached by the trajectory: 

\begin{theorem}\label{thm:diameter_tail}
There exists positive constants $c_1,C_1,c_2,C_2$ such that, for all $n$,
\[C_1 e^{-c_1 n^{1/3}}\le \P(\max_{i\ge0} \nor{Z_i}>n) \le C_2 e^{-c_2 n^{1/4}}.\]
\end{theorem}

Let us already deduce tail asymptotics for the diameter of $C(o)$:

\begin{corollary}
There exists positive constants $c_1,C_1,c_2,C_2$ such that, for all $n$,
\[C_1 e^{-c_1 n^{1/3}}\le \P(\thediam C(o)>n) \le C_2 e^{-c_2 n^{1/4}}\]
and
\[C_1 e^{-c_1 n^{1/3}}\le \P(\abs{C(o)}>n) \le C_2 e^{-c_2 n^{1/8}}.\]
In particular, $\E[\thediam C(o)]<\infty$ and $\E[\abs{C(o)}]<\infty$.
\end{corollary}

\begin{proof}[Proof of the corollary]
Since $\thediam C(o)\le\abs{C(o)}\le(\thediam C(o))^2$, it suffices to prove the bounds for the diameter. The lower bound follows immediately from the lower bound of Theorem~\ref{thm:diameter_tail}.

For the upper bound, in order to reduce from $C(o)$ to $\thepath(o)$ we start by using the size-biasing identity (Lemma~\ref{lem:size_bias}) so as to make $o$ a trap:
\begin{align*}
\P(\thediam C(o)>n)
	& = \E[\abs{C(o)}\ind{\thediam C(o)>n},\,o\in\IV].
\end{align*}
Then, if $\thediam C(o)>n$ and $o\in\IV$, there exist $v,w$ such that $v\to o$, $w\to o$ and $\nor{v-w}>n$, hence there exists $u$ (either $v$ or $w$) such that $\nor u>n/2$ and $u\to 0$:
\begin{align*}
\P(\thediam C(o)>n)
	& \le \sum_{\nor u>\frac n2}\E[\abs{C(o)}\ind{u\to o},\,o\in\IV] = \sum_{\nor u>\frac n2}\sum_v\P(u\to o,\, v\to o,\, o\in\IV)\\
	& \le \sum_{\nor u>\frac n2}\sum_v \P(u\to o)^{1/2}\P(v\to o)^{1/2}
\end{align*}
by Cauchy-Schwarz inequality. The upper bound of Theorem~\ref{thm:diameter_tail} applied to all vertices $u,v$ gives
\begin{align*}
\P(\thediam C(o)>n)
	& \le \sum_{\nor u>\frac n2}\P(u\to o)^{1/2} \sum_v \P(v\to o)^{1/2}\\
	& \le \sum_{\nor u>\frac n2}C_2^{1/2}e^{-\frac12 c_2\nor u^{1/4}} \sum_v C_2^{1/2}e^{-\frac12c_2\nor v^{1/4}}.
\end{align*}
By comparing $\nor u_\infty\le \nor u\le\sqrt d\nor u_\infty$, we have
\[\sum_{\nor u>\frac n2}e^{-\frac12 c_2\nor u^{1/4}} \le \sum_{\nor u_\infty>\frac n{2\sqrt d}}e^{-\frac12 c_2\nor u_\infty^{1/4}}= \sum_{k>\frac n{2\sqrt d}}\pa[\big]{(2k+1)^d-(2k-1)^d} e^{-\frac12 c_2 k^{1/4}},\]
and the last sum is bounded, for large $n$, by $\ds C(d)\int_{\frac n{2\sqrt d}-1}^\infty r^{d-1}e^{-\frac12c_2r^{1/4}}dr$, which is asymptotically equivalent to $C'(d)n^{d-\frac14}e^{-\frac12c_2n^{1/4}}\le C'(d)e^{-c'_2n^{1/4}}$. Finally, the above sum in $v$ in particular converges.
\end{proof}

\begin{proof}[Proof of Theorem~\ref{thm:diameter_tail}]
Let us observe that, when the walker enters a stripe of adjacent right-pushing horizontal lines, then its vertical motion is a symmetric simple random walk until it exits this stripe (and possibly conditioned by some information on the vertical lines due to the past of the walk).

\paragraph{\bf(Lower bound)}
In particular, the maximum distance of the path from $o$ is at least $n$ in the case where it keeps going (North-\ or South-)East during at least $n$ steps, which happens if, for some $\alpha>0$, the $n^\alpha$ lines above and below the origin are pointing eastward, and the exit time of an independent symmetric simple random walk out of $(-n^\alpha,n^\alpha)$ exceeds $n$:
\[\P(\max_{i\ge0}\nor{Z_i}>n)  \ge \P(\forall y\in(-n^\alpha,n^\alpha)\cap\Z,\ U_y=+1)\P(\tau_{n^\alpha}> n),\]
where, for any $L>0$, $\tau_L$ denotes the exit time out of $(-L,L)$ for a symmetric simple random walk. By Lemma~\ref{lem:SRW} in the Appendix, we get (for large $n$)
\[\P(\max_{i\ge0}\nor{Z_i}>n) \ge 2^{-(2n^\alpha-1)}e^{-c n^{1-2\alpha}},\]
for some constant $c$, which gives the lower bound of the theorem by taking the optimal value $\alpha=1/3$.

\paragraph{\bf(Upper bound)}
Up to using a union bound on all four directions, it is sufficient to consider the eastward extent, i.e.\ to show $\P(\max_{i\ge0} X_i>n)\le C e^{-c n^{1/4}}$ for some $c,C>0$ and all $n\ge0$. The following computation makes the reasoning of Proposition~\ref{prop:bounded} more quantitative, which necessitates an extra control in the transverse direction. Let us already refer the reader to Figure~\ref{fig:renewal_quantitative} depicting the forthcoming definitions.

\begin{figure}[h!]
	\includegraphics[width=15cm,page=1]{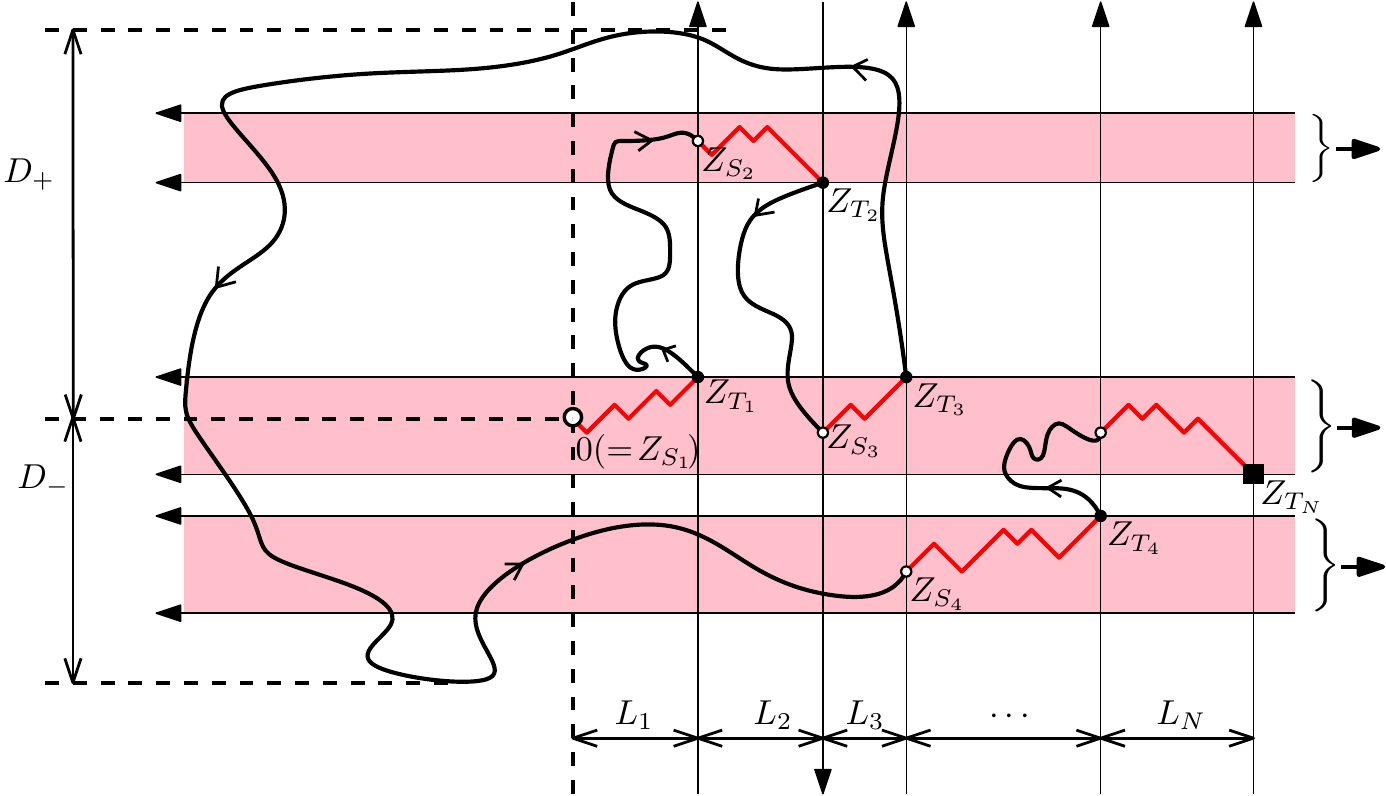}\\
\caption{Sketch for the proof of Theorem~\ref{thm:diameter_tail}. The pink stripes represent the stripes of east-going lines that are involved in the eastward discovery pieces $(Z_{S_r},\ldots,Z_{T_r})$. Each piece of path within these stripes is a one-dimensional space-time symmetric simple random walk. Note that the black pieces of path are only sketched and do not follow every constraint they should. In the case pictured here, $N=5$ and the path ends with a trap at $Z_{T_5}$; alternatively, $N=5$ could also happen with a trap encountered after $T_5$, and before $S_6$ (in that case, the rightmost vertical line would be pointing down). }\label{fig:renewal_quantitative}
\end{figure}

Let us denote by $[S_1,T_1]$, $[S_2,T_2]$,\ldots, $[S_n,T_n]$ the (maximal) intervals of times when the walk discovers new lines eastward:
\[S_1=0,\qquad T_1=\inf\{n\ge S_1\st X_{n+1}=X_n-1\},\]
and then, for all $r\ge1$,
\[S_{r+1}=\inf\{n\ge T_r\st X_{n+1}=\max_{k\le n}X_k +1\},\qquad T_{r+1}=\inf\{n\ge S_{r+1}\st X_{n+1}=X_n-1\},\]
with $S_r=\infty$ as soon as the above defining set is empty, and $T_r=\infty$ as a consequence. For $r=1,\ldots,N$, let $L_r=T_r-S_r$, and notice that, by definition, $L_r=X_{T_r}-X_{S_r}$ too.

Denote
\[N=\max\{r\ge0\st S_{r+1}<\infty\}\in\{0,1,2,\ldots\}.\]
From the reasoning of the proof of Proposition~\ref{prop:bounded}, we know that $N$ is stochastically bounded by a geometric distribution of parameter $1/2$ (corresponding the first $r$ such that the vertical line at $X_{T_r}$ is oriented in the direction that produces a trap). Denote also
\[D_-=-\min_{n\ge0} Y_n,\qquad D_+=\max_{n\ge0}Y_n,\]
and $D=\max(D_-,D_+)$.
Let $\alpha>0$, and write $f(n)=n^2$ (this choice is arbitrary within certain bounds to be specified below).
We have, for all $n$,
\begin{align*}
\P\big(\max_{i\ge0}X_i>n,\ D\le f(n)\big)
	& \le \P\Big(\sum_{i=1}^N L_i>n,\ D\le f(n)\Big)\\
	& \le \P(N>n^\alpha)+\P\bigg(N\le n^\alpha,\ \sum_{i=1}^N L_r>n,\ D\le f(n)\bigg)\\
	& \le \P(N>n^\alpha)+\sum_{r=1}^{n^\alpha}\P\Big(L_r>n^{1-\alpha},\ r\le N,\ D_r\le f(n)\Big)
\end{align*}
where $D_r=\max_{n\le S_r}|Y_n|$.

Since $N$ is stochastically dominated by a geometric distribution of parameter $1/2$,
\[\P(N>n^\alpha)\le 2^{-n^\alpha}.\]
On the other hand, given all horizontal orientations $(U_y)_{y\in\Z}$, and the past trajectory $(Z_0,\ldots,Z_{S_r})$ before time $S_r$, the sequence $(Y_{S_r},Y_{S_r+1},\ldots)$ coincides with a simple symmetric random walk starting at $Y_{S_r}$ and running until it hits a value $y\in\Z$ with $U_y=-1$. Denote by $M_n$ the maximum number of consecutive $+1$ in $(U_y)_{y\in\Z}$, starting from an index $y\in[-f(n),f(n)]$:
\[M_n=\max\{k\in\N\st \exists i\in\{-f(n)-k,\ldots,f(n)\},\ U_i=\cdots=U_{i+k-1}=+1\}.\]
Then, given $(U_y)_{y\in\Z}$ and $(Z_0,\ldots,Z_{S_r})$, on the event $\{N\le r,\ D_r\le f(n),\ M_n\le n^\gamma\}$, where $\gamma>0$, we have that $L_r$ is stochastically smaller than the exit time of a symmetric simple random walk out of $[-n^\gamma/2,n^\gamma/2]$, starting from 0 (indeed this is the worst case):
\[\P(L_r>n^{1-\alpha}\s Z_0,\ldots,Z_{S_r},\,(U_y)_{y\in\Z})\le \P(\tau_{\frac{n^\gamma}2}>n^{1-\alpha})\qquad\text{on }\{N\le r,\ D_r\le f(n),\ M_n\le n^\gamma\}.\]
As a consequence, for any $\gamma>0$,
\begin{align*}
& \P\Big(L_r>n^{1-\alpha},\ r\le N,\ D_r\le f(n)\Big)
	\\
& \le \P(M_n>n^\gamma)+\E\Big[\P\Big(L_r>n^{1-\alpha}|Z_0,\ldots,Z_{S_r},(U_y)_{y\in\Z}\Big),r\le N, D_r\le f(n),M_n\le n^\gamma\Big]\\
& \le \P(M_n>n^\gamma)+\P(\tau_{\frac{n^\gamma}2}>n^{1-\alpha}).
\end{align*}
We have, by a simple union bound,
\[\P(M_n>n^\gamma)\le(f(n)+n^\gamma)2^{-n^\gamma}\]
hence with Lemma~\ref{lem:SRW} we get (for large $n$)
\[\P\big(\max_{i\ge0}X_i>n,\ D\le f(n)\big)
\le 2^{-n^\alpha}+n^\alpha\big((f(n)+n^\gamma)2^{-n^\gamma}+\frac{n^\gamma}2 e^{-c_2 n^{1-\alpha-2\gamma}}\big).\]
Since $\log f(n)=o(n^\gamma)$, we get, for some positive constants $c,C$,
\[\P\big(\max_{i\ge0}X_i>n,\ D\le f(n)\big)
\le C e^{-c\, n^{\min(\alpha,\gamma,1-\alpha-2\gamma)}},\]
which for the optimal choice $\alpha=\gamma=\frac14$ gives
\[\P\big(\max_{i\ge0}X_i>n,\ D\le f(n)\big)\le C e^{-c\, n^{1/4}}.\]
Up to changing the constants, we may assume that the statement of the theorem is with respect to $\nor\cdot_\infty$, in which case it follows by a union bound on directions that (up to changing the above $C$),
\[\P(n< \max_i\nor{Z_i}\le f(n))\le 4\P(\max_i X_i>n,\ D\le f(n))\le C e^{-c\,n^{1/4}}.\]
In order to conclude, we decompose
\[\P(\max_i\nor{Z_i}>n)=\sum_{k=0}^\infty\P(u_k<\max_i\nor{Z_i}\le u_{k+1}) \le\sum_{k=0}^\infty C e^{-c u_k^{1/4}},\]
where $u_0=n$ and, for all $k\ge0$, $u_{k+1}=f(u_k)$. Remember we chose $f:n\mapsto n^2$, so that $u_k=n^{2^k}$ for all $k\ge0$. In particular, for $n\ge2^4$, for all $k$,
\[(u_k)^{1/4}=n^{1/4}+\big((n^{1/4})^{2^k}-n^{1/4}\big)\ge n^{1/4}+(2^{2^k}-2)\]
(indeed, $x\mapsto x^{2^k}-x$ is increasing on $[1,\infty)$) so that
\[\P(\max_i\nor{Z_i}>n)\le Ce^{-cn^{1/4}}\sum_{k=0}^\infty e^{-c(2^{2^k}-2)}=C' e^{-cn^{1/4}},\]
which concludes the proof.
\end{proof}

\section*{Appendix}

The following lemma is standard. For the sake of completeness, and for lack of a simple reference to this specific result, we sketch a short proof.

\begin{lemma}\label{lem:SRW}
Let $\tau_L$ denote the exit time out of $(-L,L)$ for a simple symmetric random walk $(S_n)_{n\ge0}$ on $\Z$ started at $0$. For all $n\in\N$ and $L\ge2$, one has
\[e^{-n\phi(L)}\le \P(\tau_L>n)\le Le^{-n\phi(L)},\]
where $\phi(L)=-\log\cos\frac\pi{2L}\sim \frac{\pi^2}{8L^2}$ as $L\to\infty$. In particular, for any $c_2<\frac{\pi^2}8<c_1$, there is $L_0$ such that, for all $n\in\N$, for all $L\ge L_0$,
\[e^{-c_1\frac n{L^2}}\le \P(\tau_L>n)\le L e^{-c_2\frac n{L^2}}.\]
\end{lemma}

\begin{proof}[Proof of Lemma~\ref{lem:SRW}]
Let us denote by $\P_k$ the law of the simple symmetric random walk $(S_n)_n$ started at $k\in\Z$.

Consider the martingale $(M_n)_{n\ge0}$ defined by $M_n=\frac{\cos(\theta S_n)}{(\cos\theta)^n}$ where $\theta=\frac\pi {2L}<\frac\pi2$ (note that $M_n$ is the real part of the classical exponential martingale $e^{i\theta S_n}/(\cos\theta)^n$) and the bounded stopping time $n\wedge\tau_L$. One has
\[\ind{n<\tau_L}\frac1L\le\ind{n<\tau_L}\sin\frac\pi{2L}=\ind{n<\tau_L}\cos\frac{\pi(L-1)}{2L}\le \cos(\frac{\pi}{2L} S_{n\wedge\tau_L})\le\ind{n<\tau_L},\]
and, by the stopping theorem,
\[1=M_0=\E_0[M_{n\wedge\tau_L}]=\E_0\Big[\cos\pa[\Big]{\frac\pi{2L} S_{n\wedge\tau_L}}\cos\pa[\Big]{\frac\pi{2L}}^{-n\wedge\tau_L}\Big],\]
hence
\[\pa[\Big]{\cos\frac\pi{2L}}^{-n}\P_0(\tau_L> n)\frac1L\le 1\le \pa[\Big]{\cos\frac\pi{2L}}^{-n}\P_0(\tau_L> n)
\]
and thus
\[\pa[\Big]{\cos\frac\pi{2L}}^n\le \P_0(\tau_L>n)\le L \pa[\Big]{\cos\frac\pi{2L}}^{n}.\]
The result follows.

%
\end{proof}

\section*{acknowledgments}
This research was supported by the Australian Research Council Grant DP180100613 and by the french ANR project MALIN (ANR-16-CE93-000).

\end{document}